\documentclass[english]{amsart}

\usepackage{latexsym}

\usepackage{amsmath}
\usepackage{rotating}
\usepackage{amsfonts}
\usepackage{amsthm}
\usepackage{amssymb}
\usepackage{verbatim}
\usepackage{psfrag}
\usepackage{graphicx}

\newtheorem{lemm}{Lemma}

\newtheorem{theo}{Theorem}

  \usepackage[all]{xy}
\xyoption{matrix}

\newcommand{\J}{\mathbf{J}}
\newcommand{\A}{\mathbf{A}}

\newcommand{\PGL}{\mathrm{PGL}}

\newcommand{\Bir}{\mathrm{Bir}}

\newcommand{\Aut}{\mathrm{Aut}}

\newcommand{\p}{\mathbb{P}}

\title{Simple relations in the Cremona group}
\author{J\'er\'emy Blanc}
\email{Jeremy.Blanc@unibas.ch}
\begin{document}
\maketitle
Let $k$ be any fixed algebraically closed field. The \emph{Cremona group} $\Bir(\p^2)$ is the group of birational transformations of the projective plane $\p^2=\p^2_{k}$.

The classical Noether-Castelnuovo Theorem says that $\Bir(\p^2)$ is generated by the group $\Aut(\p^2)\cong\PGL(3,k)$, that we will denote by $\A$, and by the \emph{standard quadratic transformation} \[\sigma\colon (X:Y:Z)\dasharrow (YZ:XZ:XY).\]

For a proof which is valid over any algebraically closed field (in particular in any characteristic), see for example \cite[Chapter V, $\S 5$, Theorem~2, page 100]{bib:Shafa}.

\bigskip

A presentation of $\Bir(\p^2)$ was given in \cite{bib:Giz}. The generators are all the quadratic transformations of the plane (among them, all elements of the form $a_1\sigma a_2$, where $a_1,a_2\in \A$), and the relations are all those of the form $q_1q_2q_3=1$ where $q_i$ is a quadratic map. The proof is quite long and uses many sophisticated tools of algebraic geometry, such as cell complexes associated to rational surfaces.

Another presentation was given in \cite{bib:Isk2} (and announced in \cite{bib:Isk1}). The surface taken here is $\p^1\times \p^1$, and the generators used are the group $\Aut(\p^1\times\p^1)$ and the de Jonqui\`eres group $\J$ of birational maps of $\p^1\times\p^1$ which preserve the first projection (see below). There is only one relation in the amalgamated product of these two groups, which is $(\rho\tau)^3=\sigma$, where $\rho=(x,y)\mapsto (x,x/y)$ and $\tau=(x,y)\mapsto(y,x)$ in local coordinates. The proof is much shorter than the one of \cite{bib:Giz}, and the number of relations is also much smaller, but everything is now on $\Bir(\p^1\times\p^1)$. There is also some gap in the proof (observed by S. Lamy): the author implicitly uses relations of the form $(\rho'\tau)^3=\sigma'$ where $\rho'$ has base-points infinitely near, without proving that they are generated by the first one (a fact not so hard to prove).

\bigskip

In this short note, we give a new presentation of the Cremona group, which are as simple as the one of \cite{bib:Isk2}, but stays on $\p^2$. The proof is also very short, and is in fact strongly inspired from the one of \cite{bib:Isk2}. We take care of infinitely near points, and translate the idea of Iskovskikh from $\p^1\times\p^1$ to $\p^2$, where it becomes simpler. We only use classical tools of plane birational geometry (base-points and blow-ups), as mathematicians of the $XIX^{th}$ century did, and as in \cite{bib:Isk2}.

\bigskip

 The {\it de Jonqui\`eres group}, that we will denote by $\J$, is the subgroup of  $\Bir(\p^2)$  consisting of elements which preserve the pencil of lines passing through $p_1=(1:0:0)$. 
This group can be viewed in local coordinates $x=X/Z$ and $y=Y/Z$ as 
 
 \[\J=\left\{\Big(x,y\Big)\dasharrow \left(\frac{ax+b}{cx+d},\frac{\alpha(x)y+\beta(x)}{\gamma(x)y+\delta(x)}\right)\ \left| \ \small\begin{array}{l} \left(\begin{array}{cc}a&b\\ c& d\end{array}\right)\in \PGL(2,k),\vspace{0.1cm}\\  \left(\begin{array}{cc}\alpha&\beta\\ \gamma& \delta\end{array}\right)\in \PGL(2,k(x))\end{array}\right.\right\}.\]
It is thus naturally isomorphic to $\PGL(2,k(x))\rtimes \PGL(2,k)$, where $\PGL(2,k)=\Aut(\p^1)$ acts on $\PGL(2,k(x))$ via its action on $k(x)=k(\p^1)$.

Since $\sigma\in \J$, the group $\Bir(\p^2)$ is generated by $\A$ and $\J$. The aim of this note is to prove the following result:

\begin{theo}\label{TheTheo}
The Cremona group $\Bir(\p^2)$ is the amalgamated product of  $\A=\Aut(\p^2)$ and $\J$ along their intersection, divided by one relation, which is 
\[\sigma\tau=\tau\sigma,\]

where $\tau\in \A$ is given by $\tau=(X:Y:Z)\mapsto (Y:X:Z)$.
\end{theo}
Since $\sigma\tau=\tau\sigma$ is easy to verify, it suffices to prove that no other relation holds. We prove this after proving the following simple lemma.
\begin{lemm}\label{Lem:RelGen}
If $\theta\in \J$ is a quadratic map having $p_1=(1:0:0)$ and $q$ as base-points, where $q$ is a proper point of $\p^2\backslash\{p_1\}$, and $\nu\in \A$ exchanges $p_1$ and $q$,  the map $\theta'=\nu \theta\nu^{-1}$ belongs to $\J$ and the relation \[\nu \theta^{-1}=(\theta')^{-1}\nu\]
is generated by the relation $\sigma \tau =\tau\sigma$ in the amalgamated product of $\A$ and $\J$.
\end{lemm}
\begin{proof}[Proof of Lemma~$\ref{Lem:RelGen}$]
The relations $\theta'=\nu \theta\nu^{-1}$ and $\nu \theta^{-1}=(\theta')^{-1}\nu$ are clearly equivalent. In particular, the result is invariant under conjugation of both $\theta$ and $\nu$ by an element of $\A\cap \J$. Choosing an element in $\A\cap \J$ which sends $q$ onto $p_2=(0:1:0)$, we can assume that $q=p_2$. Then $\nu$ is equal to $a\tau$, where $\tau=(X:Y:Z)\mapsto (Y:X:Z)$ and  $a$ is an element of $\A\cap \J$ which fixes $p_2$. We can thus assume that $\nu=\tau$. We study two cases separately, depending on the number of proper base-points of $\theta$.

$(a)$ Suppose that $\theta$ has exactly three proper base-points, which means that $\theta=a_1 \sigma a_2$ for some $a_1,a_2\in \A\cap\J$. This yields the following equality in the amalgamated product:
\[\tau \theta \tau^{-1}=\tau a_1 \sigma a_2\tau^{-1}=(\tau a_1 \tau^{-1})(\tau \sigma\tau^{-1})(\tau a_2\tau^{-1}).\]
This implies that $\tau \theta\tau^{-1}$ is equal to an element of $\J$ modulo the relation $\sigma \tau =\tau\sigma$, and yields the result.

$(b)$ Suppose now that $\theta$ has only two proper base-points, $p_1$, $p_2$, and that its third base-point, is infinitely  near to $p_i$ for some $i\in\{1,2\}$. This means that $\theta=a_1 \nu_i a_2$ for some $a_1,a_2\in \A\cap\J$, where $\nu_1$, $\nu_2$ are the following quadratic involutions:

\[\begin{array}{rcl}
\nu_1\colon& (X:Y:Z)\dasharrow &(XY:Z^2:YZ),\\
\nu_2\colon& (X:Y:Z)\dasharrow &(Z^2:XY:XZ).\end{array}\]

Denoting by $\rho_1,\rho_2\in \A\cap \J$ the maps

\[\begin{array}{rcl}
\rho_1\colon& (X:Y:Z)\dasharrow &(X:Z-Y:Z),\\
\rho_2\colon& (X:Y:Z)\dasharrow &(Z-X:Y:Z),\end{array}\]

we have $\nu_i=\rho_i\sigma\rho_i\sigma\rho_i$ in $\J$. As above, this yields the following equality:
\[\tau \theta \tau^{-1}=(\tau a_1 \tau^{-1})(\tau \rho_i\tau^{-1})(\tau\sigma\tau^{-1})(\tau\rho_i\tau^{-1})(\tau\sigma\tau^{-1})(\tau\rho_i\tau^{-1})(\tau a_2\tau^{-1}).\]

Using $\sigma \tau =\tau\sigma$ and the fact that $\tau\rho_i\tau^{-1}=\rho_{j}$ in $\A$, with $j=3-i$, we obtain 
\[\tau \sigma \tau^{-1}=(\tau a_1 \tau^{-1})(\rho_j\sigma\rho_j\sigma\rho_j)(\tau a_2\tau^{-1})=(\tau a_1 \tau^{-1})\nu_j(\tau a_2\tau^{-1}).\]

So  $\tau \theta\tau^{-1}$ is again equal to an element of $\J$ modulo the relation $\sigma \tau =\tau\sigma$. \end{proof}
\begin{proof}[Proof of Theorem~$\ref{TheTheo}$]
Taking an element $f$ in the amalgamated product $\A\star_{\A\cap \J} \J$ which corresponds to the identity map of $\Bir(\p^2)$, we have to prove that $f$ is the identity in the amalgamated product, modulo the relation $\sigma\tau=\tau\sigma$.

We write $f=j_r a_r\dots j_1a_1$ where $a_i\in \A$, $j_i\in \J$ for $i=1,\dots,n$ (maybe trivial).

We denote by $\Lambda_0$ the linear system of  lines of the plane and for $i=1,\dots,n$, we denote by $\Lambda_i$ the linear system $j_ia_i\dots j_1 a_1(\Lambda_0)$, and by $d_i$ its degree.
We define $$D=\max\Big\{d_i\Big|\ i=1,\dots,r\Big\}, n=\max\Big\{i\ \Big| d_i=D\Big\} \mbox{ and }k=\sum_{i=1}^m \Big(\deg (j_i)-1\Big).$$ When $D=1$, each $j_i$ belongs to $\A$, and the word is equal  to an element of $\A$ in the amalgamated product; since $\A$ embeds into $\Bir(\p^2)$, this case is clear. We can thus assume that $D>1$ and prove the result by induction on the pairs $(D,k)$, ordered lexicographically. 

If $j_n$ belongs to $\A$, we replace $a_{n+1}j_n a_n$ by its product in $\A$; this does not change the pair $(D,k)$ but decreases $n$ by $1$. If $j_{n+1}$ belongs to $\A$, a similar replacement decreases $r$ by $1$ without changing the pair $(D,k)$. We can thus assume that $j_n,j_{n+1}\in \J\backslash \A$ and that $a_{n+1}\in \A\backslash \J$, which means that $a_{n+1}(p_1)\not=p_1$  (recall that $p_1=(1:0:0)$ is the base-point of the pencil associated to $\J$).


The system $\Lambda_{n+1}=j_{n+1}a_{n+1}(\Lambda_n)$ has degree $d_{n+1}< d_n=D$, and $\Lambda_{n-1}=(a_n)^{-1}(j_{n})^{-1}(\Lambda_{n})$ has degree $d_{n-1}\le d_n$.
The maps $j_{n+1},j_n\in J\backslash A$ have degree $D_R$ and $D_L$ respectively, for some integers $D_R,D_L\ge 2$. The points $l_0=(a_{n+1})^{-1}(p_1)\not=p_1$ and $r_0=p_1$ are base-points of respectively $j_{n+1}a_{n+1}$ and $(a_n)^{-1}(j_{n})^{-1}$ of multiplicity $D_L-1$ and $D_R-1$. Writing $l_1,\dots,l_{2D_L-2}$ and $r_1,\dots,r_{2D_R-2}$ the other base-points of these two maps, the linear systems $\Lambda_{n+1}$ and $\Lambda_{n-1}$ have respectively degree 
\[\begin{array}{rclcl}
d_{n+1}&=&D_L\cdot d_n-(D_L-1)\cdot m(l_0)-\sum_{i}^{2D_L-2} m(l_i)&>&d_n,\\
d_{n-1}&=&D_R\cdot d_n-(D_R-1)\cdot m(r_0)-\sum_{i}^{2L_R-2}m(r_i)&\ge & d_n,\end{array}\] where $m(q)\ge 0$ is the multiplicity of a point $q$ as a base-point of $\Lambda_n$.
We order the points $l_1,\dots,l_{2D_L-2}$ so that $m(l_i)\ge m(l_{i+1})$  for each $i\ge 1$ and that if $l_i$ is infinitely near to $l_j$ then $i>j$, and we do the same for $r_1,\dots,r_{2D_R-2}$. With this order and the above inequalities, we find
\begin{equation}\label{EqIn}
\begin{array}{rcl}
 m(l_0)+ m(l_1)+m(l_2)&>&d_{n},\\
m(r_0)+m(r_1)+m(r_2)&\ge& d_{n}.\end{array}\end{equation}

$(a)$ Suppose that $m(l_0)\ge m(l_1)$ and $m(r_0)\ge m(r_1)$. We choose a point $q$ in the set $\{l_1,l_2,r_1,r_2\}\backslash \{l_0,r_0\}$ with the maximal multiplicity $m(q)$, and so that $q$ is a proper point of the plane or infinitely near to $l_0$ or $r_0$ (which are distinct proper points of the plane). We now prove that 
\begin{equation}\label{EqInn2}m(l_0)+m(r_0)+m(q)>d_n.\end{equation} 
If $l_1=r_0$, $m(q)\ge m(l_2)$ and  $m(l_0)+m(r_0)+m(q)\ge  m(l_0)+ m(l_1)+m(l_2)>d_{n}$ by $(\ref{EqIn})$. If $l_1\not=r_0$, $m(q)\ge m(l_1)\ge m(l_2)$ so $m(l_0)+m(q)>2d_n/3$. Since  $m(r_0)\ge m(r_1)\ge m(r_2)$, we have $m(r_0)\ge d_n/3$, and the inequality $(\ref{EqInn2})$ is clear. 

Because of Inequality~$(\ref{EqInn2})$, the points $l_0,r_0$ and $q$ are not aligned, and there exists a quadratic map $\theta\in \J$ with base-points $l_0,r_0,q$ (recall that $r_0=p_1$ is the point associated to the pencil of $\J$). 
Moreover, the degree of $\theta(\Lambda_n)$ is $2d_n-m(l_0)-m(r_0)-m(q)<d_n$. Recall that $a_{n+1}\in \A$ sends $l_0$ onto $r_0=p_1$. Choosing $\nu \in \A\cap \J$ which sends $a_{n+1}(r_0)$ onto $l_0$ and replacing respectively $a_{n+1}$ and $j_{n+1}$ by $\nu a_{n+1}$ and $j_{n+1}\nu^{-1}$, we can assume that $a_{n+1}$ exchanges $l_0$ and $r_0$ .   Using Lemma~\ref{Lem:RelGen}, we write $\theta'=a_{n+1}\theta (a_{n+1})^{-1}\in \J$ and obtain the following equality modulo the relation $\sigma \tau =\tau\sigma$:

\[j_{n+1}a_{n+1}j_n=j_{n+1} a_{n+1}\theta^{-1} (\theta j_n)=(j_{n+1}(\theta')^{-1})a_{n+1}(\theta j_n),\]

and both $(j_{n+1}(\theta')^{-1})$ and $(\theta j_n)$ belong to $\J$, but $a_{n+1}\in \A$.
Since $\theta(\Lambda_n)$ has degree $<d_n$, this rewriting decreases the pair $(D,k)$.

$(b)$ Suppose now that we are in a "bad case" where $m(l_0)<m(l_1)$ or $m(r_0)<m(r_1)$. We now prove that it is possible to change the writing of $f$ in the amalgamated product (modulo the relation) without changing $(D,k)$ but reversing the inequalities; we will thus be able to go back to the "good case" already studied in $(a)$ to conclude.

Assume  first that $m(r_1)>m(r_0)$. This implies that $r_1$ is a proper point of the plane, and that there exists a quadratic map $\theta\in \J$ with base-points $p_1=r_0,r_1,r_2$. Since these three points are base-points of $(j_{n})^{-1}$, the degree of $\theta j_n\in \J$ is equal to the degree of $j_n\in \J$ minus $1$.

 Taking $\nu \in \A$ which exchanges $r_0$ and $r_1$, and applying Lemma~\ref{Lem:RelGen} 
 we write $\theta'=\nu\theta \nu^{-1}\in \J$ and obtain the following equality modulo the relation $\sigma \tau =\tau\sigma$:

\[a_{n+1}j_n= (a_{n+1}\nu^{-1})\nu\theta^{-1} (\theta j_n)=(a_{n+1}\nu^{-1})(\theta')^{-1}\nu (\theta j_n),\]

and both $\theta'$ and $(\theta j_n)$ belong to $\J$, but $(a_{n+1}\nu^{-1})$ and $\nu$ belong to $\A$.
This rewriting replaces 

\[\begin{array}{rcl}
(j_1,\dots,j_{n-1},j_n,j_{n+1},\dots, j_r)&\mbox{with} & (j_1,\dots,j_{n-1},\theta j_n, (\theta')^{-1},j_{n+1},\dots,j_r),\\
(\Lambda_0,\dots,\Lambda_{n-1},\Lambda_{n},\Lambda_{n+1},\dots,\Lambda_r)&\mbox{with}&(\Lambda_0,\dots,\Lambda_{n-1},\theta(\Lambda_{n}),\nu(\Lambda_{n}) ,\Lambda_{n+1},\dots,\Lambda_r).\end{array}\]
 The degree of $\theta(\Lambda_{n})$ is equal to $2d_n- m(r_0)-m(r_1)-m(r_2)\le d_n$, and the degree of $\nu(\Lambda_n)$ is $d_n$. The new sequence has thus the same $D$, $n$ is replaced with $n+1$, and $k$ stays the same since $\deg((\theta')^{-1})-1+\deg(\theta j_n)-1=2-1+\deg(\theta j_n)-1=\deg(j_n)-1$. The system $\Lambda_n$ being replaced with $\nu(\Lambda_n)$, where $\nu\in \A$ exchanges $r_0$ and $r_1$, the multiplicity of $r_0$ as a base-point of $\nu(\Lambda_n)$ is now the biggest among the base-points of $\theta'$. In the new sequence, we have $m(r_0)>m(r_1)$ instead of $m(r_1)>m(r_0)$.

If $m(l_1)>m(l_0)$,  the same kind of replacement exchanges the points $l_0$ and $l_1$. 

We can thus go back to case $(a)$ after having made one or two replacements. This achieves the proof.
\end{proof}

\end{document}